\documentclass[11pt]{amsart}
\usepackage{graphicx}
\usepackage{oldgerm}
\usepackage{mathdots}
\usepackage{stmaryrd}
\usepackage{bm}
\usepackage[all]{xy}
\usepackage{color}
\usepackage{enumerate}

\usepackage{hyperref}
\usepackage{mathrsfs}
\usepackage{tikz} 
\usetikzlibrary{arrows,%
  decorations.markings,%
  matrix,%
  shapes}

\usepackage[latin1]{inputenc}
\usepackage{amsmath,amsthm, amscd, amssymb, amsfonts}

\numberwithin{equation}{section}

\theoremstyle{plain}

\numberwithin{equation}{section}

\newtheorem{theorem}[equation]{Theorem}
\newtheorem{corollary}[equation]{Corollary}

\newtheorem{lemma}[equation]{Lemma}

\theoremstyle{definition}
\newtheorem*{definition}{Definition}

\newtheorem{example}[equation]{Example}
\newtheorem{construction}[equation]{Construction}

\theoremstyle{definition}
\newtheorem{remark}[equation]{Remark}

\newtheorem*{conjecture}{Conjecture}

\renewcommand{\P}{\mathscr{P}}
\newcommand{\C}{\mathscr{C}}

\newcommand\Ann{\operatorname{Ann}}
\newcommand{\K}{\operatorname{\bf{K}}\nolimits}
\newcommand\Ext{\operatorname{Ext}}

\newcommand\Homol{\operatorname{H}}

\newcommand\cx{\operatorname{cx}}

\newcommand\op{\operatorname{op}}
\newcommand\ot{\otimes}
\renewcommand\mod{\operatorname{mod}}
\newcommand\Hom{\operatorname{Hom}}

\newcommand\FPdim{\operatorname{FPdim}}

\newcommand\Hoch{\operatorname{HH}}
\newcommand\HH{\Hoch}

\newcommand\unit{\mathbf{1}}

\newcommand{\DOT}{\setlength{\unitlength}{1pt}\begin{picture}(2.5,2)(1,1)\put(2,3){\circle*{2}}\end{picture}}
\newcommand{\bu}{\DOT}

\newcommand{\Coh}{\operatorname{H}\nolimits}
\newcommand{\Ho}{\operatorname{\Coh^{\bu}}\nolimits}
\newcommand{\Maxspec}{\operatorname{MaxSpec}\nolimits}

\newcommand{\s}{\Sigma}

\newcommand{\az}{\mathfrak{a}}
\newcommand{\n}{\mathfrak{n}}
\newcommand{\m}{\mathfrak{m}}

\newcommand{\diagram}[3]{\matrix (#1) [matrix of math nodes,row
  sep={#2},column sep={#3},text height=1.5ex,text
  depth=0.25ex]}

\makeatletter
\def\blx@maxline{77}
\makeatother

\begin{document}
\title[Homology of complexes]{Homology of complexes over finite tensor categories}

\author{Petter Andreas Bergh}

\address{Petter Andreas Bergh \\ Institutt for matematiske fag \\
NTNU \\ N-7491 Trondheim \\ Norway} \email{petter.bergh@ntnu.no}


\begin{abstract}
We generalize a recent result by J.F.\ Carlson to finite tensor categories having finitely generated cohomology. Specifically, we show that if the Krull dimension of the cohomology ring is sufficiently large, then there exist infinitely many non-isomorphic and nontrivial bounded complexes of projective objects, and with small homology. We also prove a version for finite dimensional algebras with finitely generated cohomology.
\end{abstract}

\maketitle

\section{Introduction}\label{sec:intro}

Let $p$ be a prime and $G$ an elementary abelian $p$-group of rank $d$, i.e.\ $G \simeq \left ( \mathbb{Z}/p \mathbb{Z} \right )^d$. In \cite{GCa1}, G.\ Carlsson conjectured that if $G$ acts freely on a nontrivial finite CW-complex $X$, then
$$\sum_{i \in \mathbb{Z}} \dim_k \Homol_i ( X; k ) \ge 2^d$$
where $k =  \mathbb{Z}/p \mathbb{Z}$. In \cite{GCa2}, he settled the case when $p=2$ and $d \le 3$, and some further isolated cases have also been confirmed in the literature. However, the conjecture remains open in general.

The algebraic version of the conjecture -- also from \cite{GCa1}  -- states that if $D$ is a nontrivial finite complex of finitely generated free $kG$-modules, then
$$\sum_{i \in \mathbb{Z}} \dim_k \Homol_i ( D ) \ge 2^d$$
 As explained (for $p=2$) in \emph{loc.\ cit.}, the algebraic version actually implies the CW-version. Moreover, in \cite{GCa2}, Carlsson also settled this version when $p=2$ and $d \le 3$. Other partial results exist here as well, for example the case when the complex $D$ has length two, in any prime characteristic; cf.\ \cite[Corollary 2.1]{AS}.

However, the algebraic version turns out not to hold in general. Namely, in \cite{IW}, S.\ Iyengar and M.\ Walker provided a counterexample when $p$ is odd and $d \ge 8$. The construction they provide is based on properties of exterior algebras, in particular the existence of a so-called Lefschetz element when there are an even number of generators. It is natural to ask if one can use this construction to obtain counterexamples to the CW-version of Carlsson's conjecture, but in \cite{RS} it was shown that this is not possible.

In the recent paper \cite{Ca}, J.\ Carlson extends the counterexamples from \cite{IW} to arbitrary groups. More precisely, for a field $k$ of odd characteristic $p$, and \emph{any} finite group $G$ of $p$-rank $d$ at least $8$, he shows that there exist infinitely many non-isomorphic and nontrivial finite complexes $D$ of projective $kG$-modules, with $\sum_{i \in \mathbb{Z}} \dim_k \Homol_i ( D ) < 2^d$. Of course, when $G$ is a $p$-group, then the projective $kG$-modules are free, and so Carlson's result provides a whole range of counterexamples to Carlsson's algebraic conjecture. Moreover, the counterexamples from \cite{IW} are recovered; the complexes constructed by Iyengar and Walker are quasi-isomorphic to complexes constructed in \cite{Ca}.

In this paper, we prove an analogous version of Carlson's result for finite tensor categories. For such a category $\left ( \C, \ot, \unit \right )$ over a ground field $k$, one can not use the vector space dimension to measure the size of the homology of a complex, since the homology is itself an object of $\C$, and therefore in general not a vector space. Instead, we measure the homology in terms of an arbitrary nontrivial additive function on $\C$. In other words, we use a function $f$ that assigns to every object $M \in \C$ a non-negative real number $f(M)$, in such a way that $f$ is additive on short exact sequences, and $f( \unit ) \neq 0$. We show that when the characteristic of $k$ is odd, and $\C$ is sufficiently nice with a cohomology ring of Krull dimension $d$ at least $8$, then there exist infinitely many non-isomorphic and nontrivial bounded complexes $D$ of projective objects, with 
$$\frac{1}{f ( \unit )} \sum_{i \in \mathbb{Z}} f ( \Homol_i(D) ) < 2^d$$
Two important such additive functions are the Frobenius-Perron dimension, and the length function. Given any object $M \in \C$, the inequality $\ell (M) \le \FPdim_{\C}(M)$ holds, and $\ell ( \unit ) = \FPdim_{\C}( \unit ) =1$. Thus, when $\C$ is as above, then there exist infinitely many non-isomorphic and nontrivial bounded complexes $D$ of projective objects, with 
$$\sum_{i \in \mathbb{Z}} \ell ( \Homol_i(D) ) \le \sum_{i \in \mathbb{Z}} \FPdim_{\C} ( \Homol_i(D) ) < 2^d$$
In particular, we can apply this result to suitably nice finite dimensional Hopf algebras over $k$, since in this case the Frobenius-Perron dimension of a finitely generated module coincides with the vector space dimension. Thus for such a Hopf algebra, there exist infinitely many non-isomorphic and nontrivial bounded complexes $D$ of projective modules, with 
$$\sum_{i \in \mathbb{Z}} \dim_k \Homol_i(D) < 2^d$$
Carlson's result from \cite{Ca} is a particular case of this last result, applied to the finite dimensional Hopf algebra $kG$.

We also prove a version of the main result for finite dimensional algebras having finitely generated cohomology. The cohomology ring we work with is then the Hochschild cohomology ring, which, like the cohomology ring of a finite tensor category, is graded-commutative. When the cohomology is finitely generated, and the algebra admits a module of sufficiently large complexity, the theorem shows that there exist infinitely many non-isomorphic and nontrivial bounded complexes of projective modules, with small homology. The size of the homology modules is measured by the vector space dimension.

\subsection*{Acknowledgments}
I would like to thank Jon F.\ Carlson for very helpful comments and conversations on this topic. I would also like to thank the organizers of the Representation Theory program hosted by the Centre for Advanced Study at The Norwegian Academy of Science and Letters, where I spent parts of fall 2022.

\section{Preliminaries}\label{sec:prelim}

Let us fix a field $k$ -- not necessarily algebraically closed -- together with a finite tensor $k$-category $\left ( \C, \ot, \unit \right )$. Thus $\C$ is a locally finite $k$-linear abelian category, with a finite set of isomorphism classes of simple objects. There are enough projective objects, and every object admits a projective cover, and therefore a minimal projective resolution. Moreover, there is an associative (up to functorial isomorphisms) bifunctor
\begin{center}
\begin{tikzpicture}
\diagram{d}{3em}{3em}{
 \C \times \C & \C \\
 };
\path[->, font = \scriptsize, auto]
(d-1-1) edge node{$\ot$} (d-1-2);
\end{tikzpicture}
\end{center}
called the tensor product, which is compatible with the abelian structure of $\C$, together with a unit object $\unit \in \C$ (with respect to $\ot$) which is simple as an object of $\C$. Finally, the category $\C$ is rigid, meaning that all objects have left and right duals.

A typical example is the category of finitely generated left $kG$-modules, where $G$ is a finite group. In this case, the finite tensor category is symmetric, meaning that the tensor product is commutative in a certain strong sense. More generally, if $A$ is a finite dimensional Hopf algebra, then the category of finitely generated left $A$-modules is a finite tensor category, with a non-commutative tensor product in general. Recently, a very interesting class of symmetric finite tensor categories were introduced in \cite{BE1, BEO, C}.

\begin{remark}\label{rem:properties}
(1) Finite tensor categories enjoy a whole range of important properties. First of all, since the underlying abelian category is locally finite -- that is, the morphism spaces are finite dimensional and every object has finite length -- the Krull-Schmidt Theorem and the Jordan-H{\"o}lder Theorem hold, cf.\ \cite[Section 1.5]{EGNO}. In other words, every object decomposes uniquely (up to isomorphism) as a finite direct sum of indecomposable objects, and admits a Jordan-H{\"o}lder series with unique simple multiplicities. Secondly, the existence of dual objects has far reaching consequences. For example, by \cite[Proposition 4.2.1 and Proposition 4.2.12]{EGNO}, it implies that the tensor product $\ot$ is bi-exact, and that the projective objects form a two-sided ideal: the tensor product between a projective object and any other object is again projective. Moreover, by \cite[Proposition 6.1.3 and Remark 6.1.4]{EGNO}, the existence of duals implies that the category is quasi-Frobenius, that is, the projective and injective objects coincide.

(2) Since $\unit$ is a simple object, the $k$-algebra $\Hom_{\C}( \unit, \unit )$ is a division ring. Moreover, this ring is actually commutative (see below), and so it is a finite field extension of $k$. In particular, when $k$ is algebraically closed, then $\Hom_{\C}( \unit, \unit ) = k$.
\end{remark}

Given objects $M,N \in \C$, we denote by $\Ext_{\C}^*(M,N)$ the graded $k$-vector space $\oplus_{n=0}^{\infty} \Ext_{\C}^n(M,N)$. When $M=N$, this is a graded $k$-algebra with Yoneda composition as multiplication. The cohomology algebra $\Ext_{\C}^*( \unit, \unit )$ of the unit object $\unit$ is denoted by $\Coh^*( \C )$; this is the \emph{cohomology ring} of $\C$, and it is graded-commutative by \cite[Theorem 1.7]{SA}. The tensor product induces a homomorphism
\begin{center}
\begin{tikzpicture}
\diagram{d}{3em}{4em}{
 \Coh^*( \C ) & \Ext_{\C}^*(M,M) \\
 };
\path[->, font = \scriptsize, auto]
(d-1-1) edge node{$- \ot M$} (d-1-2);
\end{tikzpicture}
\end{center}
of graded $k$-algebras, making $\Ext_{\C}^*(M,N)$ both a left and a right module over $\Coh^*( \C )$, via $- \ot N$ and $- \ot M$ followed by Yoneda composition. By modifying the proof of \cite[Theorem 1.1]{SS04}, one can show that for elements $\eta \in \Coh^m( \C )$ and $\theta \in \Ext_{\C}^n(M,N)$, there is an equality
$$( \eta \ot N ) \circ \theta = (-1)^{mn} \theta \circ ( \eta \ot M )$$
where the symbol $\circ$ denotes Yoneda composition. In other words, the left and the right scalar actions of $\Coh^*( \C )$ on $\Ext_{\C}^*(M,N)$ coincide up to a sign, when we consider homogeneous elements. 

In \cite{EO}, Etingof and Ostrik conjectured the following, which is still open:
\begin{conjecture}
\sloppy The cohomology ring $\Coh^*( \C )$ is finitely generated, and $\Ext_{\C}^*(M,M)$ is a finitely generated $\Coh^*( \C )$-module for all objects $M \in \C$.
\end{conjecture}
Note that when $\C$ has finitely generated cohomology as in the conjecture, then for all objects $M,N \in \C$, the $\Coh^* ( \C )$-module $\Ext_{\C}^*(M,N)$ is finitely generated, and not just the two modules $\Ext_{\C}^*(M,M)$ and $\Ext_{\C}^*(N,N)$. Namely, the $\Coh^* ( \C )$-module $\Ext_{\C}^*(M \oplus N,M \oplus N)$ is finitely generated, and it has $\Ext_{\C}^*(M,N)$ as a direct summand. 

\begin{definition}
We say that the finite tensor category $\C$ satisfies the \emph{finiteness condition} \textbf{Fg} if the cohomology ring $\Coh^*( \C )$ is finitely generated, and $\Ext_{\C}^*(M,M)$ is a finitely generated $\Coh^*( \C )$-module for all objects $M \in \C$.
\end{definition}

When $\C$ is the category of finitely generated left $kG$-modules for a finite group $G$, then \textbf{Fg} holds by a classical result of Evens and Venkov; cf.\ \cite{E, V1, V2}. However, the conjecture remains open for module categories of arbitrary finite dimensional Hopf algebras. Recently, Benson and Etingof proved in \cite{BE2} that \textbf{Fg} holds for the symmetric finite tensor categories introduced in \cite{BE1, BEO, C}.

When a finite tensor category has finitely generated cohomology, then there is a rich theory of cohomological support varieties. Define $\Ho ( \C )$ to be $\Coh^{2*}( \C )$ when the characteristic of $k$ is not $2$, and $\Coh^{*}( \C )$ when the characteristic is $2$. Since $\Coh^*( \C )$ is graded-commutative, the algebra $\Ho ( \C )$ is commutative in the ordinary sense. Moreover, when the characteristic of $k$ is not $2$, then the odd degree homogeneous elements of $\Coh^*( \C )$ are nilpotent. Therefore, when \textbf{Fg} holds, then $\Ho ( \C )$ is finitely generated, and $\Ext_{\C}^*(M,M)$ is a finitely generated $\Ho ( \C )$-module for all objects $M \in \C$.

We denote the unique homogeneous maximal ideal of $\Ho ( \C )$ by $\m_0$, i.e.\ $\m_0 = \Ho ( \C ) \cap \Coh^{\ge 1}( \C )$; recall from Remark \ref{rem:properties}(2) that $\Coh^0( \C )$ is a field. The set of all maximal ideals of $\Ho ( \C )$ is denoted by $\Maxspec \Ho ( \C )$. The \emph{support variety} of an object $M \in \C$ is now defined as
$$V_{\C} (M) = \{ \m_0 \} \cup \{ \m \in \Maxspec \Ho ( \C ) \mid \Ann_{\Ho ( \C )} \Ext_{\C}^*(M,M) \subseteq \m \}$$
When \textbf{Fg} holds, then these support varieties encode important homological information on the objects, as shown in \cite{BPW1, BPW2}. Many of the classical results on support varieties for modules over group algebras carry over to this setting.

\section{The main result}\label{sec:main}

For a finite tensor category $\C$, we denote by $\K^b ( \C )$ the triangulated category of bounded complexes over $\C$, with its usual suspension $\s \colon \K^b ( \C ) \longrightarrow \K^b ( \C )$ that shifts a complex one step to the left and changes the sign on the differentials. The distinguished triangles are defined in terms of mapping cones of chain maps. The tensor product $\ot$ induces a monoidal structure on $\K^b ( \C )$, compatible with the triangulated structure in such a way that $\K^b ( \C )$ is a \emph{triangulated tensor category} in the sense of \cite{BKSS}, that is, a possibly noncommutative tensor triangulated category. The tensor product is exact; when we tensor a distinguished triangle with an object, the result is another distinguished triangle. In this process, we implicitly use some particular isomorphisms. Namely, given complexes $X,Y \in \K^b ( \C )$, there are canonical isomorphisms
\begin{center}
\begin{tikzpicture}
\diagram{d}{3em}{3em}{
 \s \left ( X \ot Y \right ) & \s X \ot Y &  \s \left ( X \ot Y \right ) & X \ot \s Y \\
 };
\path[->, font = \scriptsize, auto]
(d-1-1) edge node{$\lambda_{X,Y}$} (d-1-2)
(d-1-3) edge node{$\rho_{X,Y}$} (d-1-4);
\end{tikzpicture}
\end{center}
natural in $X$ and $Y$, with the property that the diagram
\begin{center}
\begin{tikzpicture}
\diagram{d}{3em}{3em}{
 \s^2 \left ( X \ot Y \right ) & \s \left ( \s X \ot Y \right ) \\
 \s \left ( X \ot \s Y \right ) & \s X \ot \s Y \\
 };
\path[->, font = \scriptsize, auto]
(d-1-1) edge node{$\s \lambda_{X,Y}$} (d-1-2)
(d-2-1) edge node{$\lambda_{X, \s Y}$} (d-2-2)
(d-1-1) edge node{$\s \rho_{X,Y}$} (d-2-1)
(d-1-2) edge node{$\rho_{\s X,Y}$} (d-2-2);
\end{tikzpicture}
\end{center}
anticommutes. Using these, one obtains for all integers $m$ and $d \ge 2$, and all complexes $X_1, \dots, X_d$ in $\K^b ( \C )$, a natural isomorphism
\begin{center}
\begin{tikzpicture}
\diagram{d}{3em}{3em}{
 \s^m \left ( X_1 \ot \cdots \ot X_d \right ) & X_1 \ot \cdots \ot \s^m X_i  \ot \cdots \ot X_d \\
 };
\path[->, font = \scriptsize, auto]
(d-1-1) edge node{$\sigma^m_{d,i}$} (d-1-2);
\end{tikzpicture}
\end{center}
for each $1 \le i \le d$. Moreover, given another integer $n$, the diagram
\begin{center}
\begin{tikzpicture}
\diagram{d}{3em}{3em}{
  \s^{m+n} \left ( X_1 \ot \cdots \ot X_d \right ) & \s^n \left ( X_1 \ot \cdots \ot \s^m X_i  \ot \cdots \ot X_d \right ) \\
\s^m \left ( X_1 \ot \cdots \ot \s^n X_j  \ot \cdots \ot X_d \right ) & X_1 \ot \cdots \ot \s^m X_i  \ot \cdots \ot \s^n X_j  \ot \cdots \ot X_d \\
 };
\path[->, font = \scriptsize, auto]
(d-1-1) edge node{$\s^n \sigma^m_{d,i}$} (d-1-2)
(d-2-1) edge node{$\sigma^m_{d,i}$} (d-2-2)
(d-1-1) edge node{$\s^m \sigma^n_{d,j}$} (d-2-1)
(d-1-2) edge node{$\sigma^n_{d,j}$} (d-2-2);
\end{tikzpicture}
\end{center}
anticommutes with a factor $(-1)^{mn}$ whenever $i < j$.

We now construct some specific complexes and chain maps in $\K^b ( \C )$ that we shall use in the proof of the main result. The construction is analogous to that presented in \cite[Section 3]{BC}.

\begin{construction}\label{con:complex}
(1) Let $\zeta$ be a nonzero homogeneous element of $\Coh^*( \C )$, of degree $n \ge 2$, and fix a minimal projective resolution $(P_i, d_i)$ of the unit object $\unit$. This element corresponds to a nonzero morphism $\hat{\zeta} \colon \Omega_{\C}^n ( \unit ) \longrightarrow \unit$, where $\Omega_{\C}^n ( \unit )$ is the $n$th syzygy of $\unit$, that is, the image of the morphism $d_n$. From this morphism we obtain a commutative pushout diagram
\begin{center}
\begin{tikzpicture}
\diagram{d}{2.8em}{2.8em}{
0 & \Omega_{\C}^n ( \unit ) & P_{n-1} & P_{n-2} & \cdots & P_0 & \unit & 0 \\
0 & \unit & K_{\zeta} & P_{n-2} & \cdots & P_0 & \unit & 0 \\
 };
\path[->, font = \scriptsize, auto]
(d-1-1) edge (d-1-2) 
(d-1-2) edge (d-1-3) 
(d-1-3) edge node{$d_{n-1}$} (d-1-4)
(d-1-4) edge node{$d_{n-2}$} (d-1-5)
(d-1-5) edge node{$d_1$} (d-1-6)
(d-1-6) edge node{$d_0$} (d-1-7)
(d-1-7) edge (d-1-8)
(d-2-1) edge (d-2-2) 
(d-2-2) edge node{$\mu$} (d-2-3) 
(d-2-3) edge node{$\rho$} (d-2-4)
(d-2-4) edge node{$d_{n-2}$} (d-2-5)
(d-2-5) edge node{$d_1$} (d-2-6)
(d-2-6) edge node{$d_0$} (d-2-7)
(d-2-7) edge (d-2-8)
(d-1-2) edge node{$\hat{\zeta}$} (d-2-2)
(d-1-3) edge (d-2-3)
(d-1-4) edge node{$1$} (d-2-4)
(d-1-6) edge node{$1$} (d-2-6)
(d-1-7) edge node{$1$} (d-2-7);
\end{tikzpicture}
\end{center}
with exact rows, and in turn a complex
\begin{center}
\begin{tikzpicture}
\diagram{d}{2.8em}{2.8em}{
\cdots & 0 & K_{\zeta} & P_{n-2} & \cdots & P_0 & 0 & \cdots \\
 };
\path[->, font = \scriptsize, auto]
(d-1-1) edge (d-1-2) 
(d-1-2) edge (d-1-3) 
(d-1-3) edge node{$\rho$} (d-1-4) 
(d-1-4) edge node{$d_{n-2}$} (d-1-5) 
(d-1-5) edge node{$d_1$} (d-1-6)
(d-1-6) edge (d-1-7)
(d-1-7) edge (d-1-8);
\end{tikzpicture}
\end{center}
of length $n$, which we denote by $C_{\zeta}$ (if $n=2$, then $d_1 = \rho$ in this complex). 

(2) The composition $\mu \circ d_0$ induces a chain map $\nu \colon \s^{n-1} C_{\zeta} \longrightarrow C_{\zeta}$, displayed as
\begin{center}
\begin{tikzpicture}
\diagram{d}{2.8em}{3.5em}{
\cdots &P_1 & P_0 & 0 & \cdots \\
\cdots & 0 & K_{\zeta} & P_{n-2} & \cdots \\
 };
\path[->, font = \scriptsize, auto]
(d-1-1) edge (d-1-2) 
(d-1-2) edge node{$(-1)^{n-1}d_1$} (d-1-3) 
(d-1-3) edge (d-1-4)
(d-1-4) edge (d-1-5)
(d-2-1) edge (d-2-2) 
(d-2-2) edge (d-2-3) 
(d-2-3) edge node{$\rho$} (d-2-4)
(d-2-4) edge node{$d_{n-2}$} (d-2-5)
(d-1-2) edge (d-2-2)
(d-1-3) edge node{$\mu \circ d_0$} (d-2-3)
(d-1-4) edge (d-2-4);
\end{tikzpicture}
\end{center}
with the following convention if $n=2$: $d_{n-2} = 0$, $d_1 = \rho$, and $P_1 = K_{\zeta}$.
\end{construction}

Recall from Section \ref{sec:prelim} that when $\C$ is a finite tensor category over a field of characteristic $p \neq 2$, then $\Ho ( \C )$ denotes the even part of the cohomology ring $\Coh^*( \C )$, i.e.\ $\Ho ( \C ) = \Coh^{2*}( \C )$. When $p=2$, then $\Ho ( \C )$ denotes just $\Coh^*( \C )$ itself. Moreover, the unique homogeneous maximal ideal of $\Ho ( \C )$ is denoted by $\m_0$. If $\C$ satisfies \textbf{Fg}, and the Krull dimension of $\Ho ( \C )$ is $d$, say, then a sequence $\zeta_1, \dots, \zeta_d$ of nonzero homogeneous elements (of positive degrees) in $\Ho ( \C )$ is called a \emph{homogeneous system of parameters} if the ideal $( \zeta_1, \dots, \zeta_d )$ is $\m_0$-primary, so that its radical is $\m_0$. In the following lemma, we show that for such a sequence, the object $K_{\zeta_1} \ot \cdots \ot K_{\zeta_d}$ is projective in $\C$, where $K_{\zeta_i}$ is the object appearing in the complex $C_{\zeta_i}$ from Construction \ref{con:complex}.

\begin{lemma}\label{lem:projective}
Let $k$ be a field and $\left ( \C, \ot, \unit \right )$ a finite tensor $k$-category. Suppose that $\C$ satisfies \emph{\textbf{Fg}}, and that $\zeta_1, \dots, \zeta_d$ is a homogeneous system of parameters in $\Ho ( \C )$, where $d$ is the Krull dimension of $\Ho ( \C )$. Then the object $K_{\zeta_1} \ot \cdots \ot K_{\zeta_d}$ is projective.
\end{lemma}

\begin{proof}
As mentioned in Construction \ref{con:complex}, each $\zeta_i$ corresponds to a nonzero morphism $\hat{\zeta}_i \colon \Omega_{\C}^n ( \unit ) \longrightarrow \unit$. This morphism must necessarily be an epimorphism, since $\unit$ is a simple object; we denote its kernel by $L_{\zeta_i}$. By \cite[Theorem 5.2]{BPW1}, the equality $V_{\C} ( L_{\zeta_i} \ot M ) = V_{\C} (M) \cap Z( \zeta_i )$ holds for every object $M \in \C$, where $Z( \zeta_i )$ is the set of all $\m \in \Maxspec \Ho ( \C )$ containing $\zeta_i$. Now since $K_{\zeta_i}$ is the pushout of the epimorphism $\hat{\zeta}_i$ and the monomorphism $\Omega_{\C}^n ( \unit ) \longrightarrow P_{n-1}$, there is an exact commutative diagram
\begin{center}
\begin{tikzpicture}
\diagram{d}{2.5em}{2.5em}{
& 0 & 0 \\
& L_{\zeta_i} & L_{\zeta_i} \\
0 & \Omega_{\C}^n ( \unit ) & P_{n-1} & \Omega_{\C}^{n-1} ( \unit ) & 0 \\
0 & \unit & K_{\zeta_i} & \Omega_{\C}^{n-1} ( \unit ) & 0 \\
& 0 & 0 \\
 };
\path[->, font = \scriptsize, auto]
(d-2-2) edge node{$1$} (d-2-3) 
(d-3-1) edge (d-3-2) 
(d-3-2) edge (d-3-3) 
(d-3-3) edge (d-3-4)
(d-3-4) edge (d-3-5)
(d-4-1) edge (d-4-2) 
(d-4-2) edge node{$\mu_i$}  (d-4-3) 
(d-4-3) edge (d-4-4)
(d-4-4) edge (d-4-5)
(d-1-2) edge (d-2-2) 
(d-2-2) edge (d-3-2) 
(d-3-2) edge node{$\hat{\zeta}_i$} (d-4-2)
(d-4-2) edge (d-5-2)
(d-1-3) edge (d-2-3) 
(d-2-3) edge (d-3-3) 
(d-3-3) edge (d-4-3)
(d-4-3) edge (d-5-3)
(d-3-4) edge node{$1$} (d-4-4);
\end{tikzpicture}
\end{center}
It follows that the object $K_{\zeta_i}$ has the same property with respect to support varieties as the object $L_{\zeta_i}$. For if $M$ is any object, then by applying $- \ot M$ to the right-hand vertical exact sequence, we obtain an exact sequence
\begin{center}
\begin{tikzpicture}
\diagram{d}{3em}{3em}{
0 & L_{\zeta_i} \ot M & P_{n-1} \ot M & K_{\zeta_i} \ot M & 0 \\
 };
\path[->, font = \scriptsize, auto]
(d-1-1) edge (d-1-2)
(d-1-2) edge (d-1-3)
(d-1-3) edge (d-1-4)
(d-1-4) edge (d-1-5);
\end{tikzpicture}
\end{center}
Since the object $P_{n-1} \ot M$ is projective, it follows from \cite[Proposition 3.3(vii)]{BPW1} that the objects $L_{\zeta_i} \ot M$ and $K_{\zeta_i} \ot M$ have the same support varieties, hence $V_{\C} ( K_{\zeta_i} \ot M ) = V_{\C} (M) \cap Z( \zeta_i )$. By repeatedly applying this equality, and using the fact that $V_{\C} ( \unit ) = \Maxspec \Ho ( \C )$, we obtain
\begin{eqnarray*}
V_{\C} \left ( K_{\zeta_1} \ot \cdots \ot K_{\zeta_d} \right ) & = & V_{\C} \left ( K_{\zeta_1} \ot \cdots \ot K_{\zeta_d} \ot \unit \right ) \\
& = & V_{\C} \left ( \unit \right ) \cap Z \left ( \zeta_1 \right ) \cap \cdots \cap Z \left ( \zeta_d \right ) \\
& = & Z \left ( \zeta_1, \dots, \zeta_d \right ) \\
& = & \{ \m_0 \}
\end{eqnarray*}
since the radical of the ideal $( \zeta_1, \dots, \zeta_d )$ is $\m_0$. Consequently, the object $K_{\zeta_1} \ot \cdots \ot K_{\zeta_d}$ is projective, by \cite[Corollary 4.3]{BPW1}.
\end{proof}

We now prove the main result: when \textbf{Fg} holds and the Krull dimension of the cohomology ring is sufficiently large, then there exist infinitely many non-isomorphic and nontrivial bounded complexes of projective objects, and with small homology. By an \emph{additive function} on $\C$ we mean a group homomorphism from the Grothendieck group of $\C$ into the non-negative real numbers. Equivalently, such a function $f$ assigns to each object $M \in \C$ a real number $f(M) \ge 0$, with $f(M) = f(L) + f(N)$ whenever there exists a short exact sequence
\begin{center}
\begin{tikzpicture}
\diagram{d}{3em}{3em}{
0 & L & M & N & 0 \\
 };
\path[->, font = \scriptsize, auto]
(d-1-1) edge (d-1-2)
(d-1-2) edge (d-1-3)
(d-1-3) edge (d-1-4)
(d-1-4) edge (d-1-5);
\end{tikzpicture}
\end{center}
in $\C$. We call $f$ \emph{nontrivial} if $f( \unit ) \neq 0$.

\begin{theorem}\label{thm:main}
Let $k$ be a field of characteristic not $2$, and $\left ( \C, \ot, \unit \right )$ a finite tensor $k$-category. Suppose that $\C$ satisfies \emph{\textbf{Fg}}, and that the Krull dimension $d$ of $\Ho ( \C )$ is at least $8$. Then given any nontrivial additive function $f$ on $\C$, there exist infinitely many non-isomorphic and nontrivial bounded complexes $D \in \K^b ( \C )$ of projective objects, with 
$$\frac{1}{f ( \unit )} \sum_{i \in \mathbb{Z}} f ( \Homol_i(D) ) < 2^d$$
\end{theorem}

\begin{proof}
Given any positive real number $\alpha$, the function $\alpha f$ is also nontrivial and additive on $\C$. We may therefore suppose that $f( \unit ) = 1$.

By the graded version of the Noether Normalization Lemma (cf.\ \cite[Theorem 1.5.17]{BH}), there exists a homogeneous system of parameters $\zeta_1, \dots, \zeta_d$ in $\Ho ( \C )$. Given any sequence $n_1, \dots, n_d$ of positive integers, the ideal $( \zeta_1^{n_1}, \dots, \zeta_d^{n_d} )$ is also $\m_0$-primary, so we may assume without loss of generality that the elements $\zeta_1, \dots, \zeta_d$ are all of the same (even) degree in $\Ho ( \C )$, say $n$.

For each $1 \le i \le d$, consider the complex $C_{\zeta_i}$ of length $n$, together with the chain map $\nu_i \colon \s^{n-1} C_{\zeta_i} \longrightarrow C_{\zeta_i}$, from Construction \ref{con:complex}. Since $\mu_i$ is a monomorphism and $d_0$ is nonzero, the composition $\mu_i \circ d_0$ is nonzero, and so the chain map $\nu_i$ is not null-homotopic. Therefore (the homotopy equivalence class of) $\nu_i$ is a nonzero morphism in $\Hom_{\K^b ( \C )}( \s^{n-1} C_{\zeta_i}, C_{\zeta_i} )$, with $\nu_i^2 = 0$ in $\Hom_{\K^b ( \C )}^*( C_{\zeta_i}, C_{\zeta_i} )$. Moreover, note that the homology of $C_{\zeta_i}$ is nonzero in precisely two degrees: $\Homol_0(C_{\zeta_i}) \simeq \Homol_{n-1}(C_{\zeta_i}) \simeq \unit$. Consequently, we see that $\nu_i$ induces an isomorphism $\bar{\nu}_i \colon \Homol_0(C_{\zeta_i}) \longrightarrow \Homol_{n-1}(C_{\zeta_i})$.

To make the notation in the rest of the proof a bit easier, we set $m = n-1$; thus $m$ is an odd integer. Consider now the complex $C = C_{\zeta_1} \ot \cdots \ot C_{\zeta_d}$ in $\K^b ( \C )$ of length $dm+1$, with nonzero objects in degrees $0, 1, \dots, dm$. In degrees $0,1, \dots, dm-1$, the objects are projective, since they are direct sums of tensor products in which at least one of the $P_j$ appears as a factor (recall that the projective objects of $\C$ form a two-sided ideal). The object in degree $dm$ is the tensor product $K_{\zeta_1} \ot \cdots \ot K_{\zeta_d}$, which also is projective, as we showed in Lemma \ref{lem:projective}. Thus $C$ is a complex of projective objects. Now for each $i$, the chain map $\nu_i$ induces a chain map $\theta_i \colon \s^m C \longrightarrow C$, given by the composition
\begin{center}
\begin{tikzpicture}
\diagram{d}{3em}{3em}{
\s^m C & C_{\zeta_1} \ot \cdots \ot \s^m C_{\zeta_i} \ot \cdots \ot C_{\zeta_d} & C \\
 };
\path[->, font = \scriptsize, auto]
(d-1-1) edge node{$\sigma^m_{d,i}$} (d-1-2)
(d-1-2) edge node{$\hat{\nu}_i$} (d-1-3);
\end{tikzpicture}
\end{center}
where $\sigma^m_{d,i}$ is the natural isomorphism, and $\hat{\nu}_i = 1 \ot \cdots \ot \nu_i \ot \cdots \ot 1$. Since $\nu_i$ squares to zero in $\Hom_{\K^b ( \C )}^*( C_{\zeta_i}, C_{\zeta_i} )$, we see that $\theta_i^2 =0$ in $\Hom_{\K^b ( \C )}^*( C, C )$. Moreover, when $i < j$ we obtain
\begin{eqnarray*}
\theta_i \cdot \theta_j & = & \theta_i \circ \s^m \left ( \theta_j \right ) \\
& = & \hat{\nu}_i \circ \sigma^m_{d,i} \circ \s^m \hat{\nu}_j \circ \s^m \sigma^m_{d,j} \\
& = & \hat{\nu}_i \circ \hat{\nu}_j \circ \sigma^m_{d,i} \circ \s^m \sigma^m_{d,j} \\
& = & (-1)^{m^2} \hat{\nu}_j \circ \hat{\nu}_i \circ \sigma^m_{d,j} \circ \s^m \sigma^m_{d,i} \\
& = & - \hat{\nu}_j \circ \sigma^m_{d,j} \circ \s^m \hat{\nu}_i \circ  \s^m \sigma^m_{d,i} \\
& = & - \theta_j \cdot \theta_i
\end{eqnarray*}
Here we have used the fact that $\sigma^m_{d,i}$ and $\sigma^m_{d,j}$ are natural, so that they commute with $\hat{\nu}_j$ and $\hat{\nu}_i$. Moreover, we have used that $ \sigma^m_{d,i} \circ \s^m \sigma^m_{d,j} = (-1)^{m^2} \sigma^m_{d,j} \circ \s^m \sigma^m_{d,i}$ -- as explained in the beginning of this section -- and the fact that $m$ is odd. Finally, we have used that $\hat{\nu}_i$ and $\hat{\nu}_j$ commute.

The above shows that the chain maps $\theta_1, \dots, \theta_d$ generate a $d$-fold exterior algebra $\Lambda \subseteq \Hom_{\K^b ( \C )}^*( C, C )$, concentrated in degrees $0,m,2m, \dots, dm$. How do they act on the homology of $C$? Since each $C_{\zeta_i}$ has only two nonzero homology objects $\Homol_0(C_{\zeta_i})$ and $\Homol_m(C_{\zeta_i})$, and the tensor product is exact in each variable, it follows from the K{\"u}nneth Theorem that $C$ has homology only in degrees $0,m,2m, \dots, dm$, with
$$\Homol_{tm} \left ( C \right ) \simeq \bigoplus_{i_1+ \cdots + i_d = t} \Homol_{mi_1} \left ( C_{\zeta_1} \right ) \ot \cdots \ot \Homol_{mi_d} \left ( C_{\zeta_d} \right )$$
Moreover, since $\Homol_0(C_{\zeta_i}) \simeq \Homol_m(C_{\zeta_i}) \simeq \unit$, we see that $\Homol_{tm}(C)$ is the direct sum of ${d}\choose{t}$
copies of $\unit$; let us denote these by $\unit_{i_1, \dots, i_d}$, with $t$ of the $i_1, \dots, i_d$ equal to $m$, and the other ones equal to zero. Thus
\begin{eqnarray*}
\Homol_0 \left ( C \right ) & \simeq & \unit_{0, \dots, 0} \\
\Homol_m \left ( C \right ) & \simeq & \unit_{m,0, \dots, 0} \oplus \unit_{0,m, \dots, 0} \oplus \cdots \oplus \unit_{0, \dots, 0,m} \\
\Homol_{2m} \left ( C \right ) & \simeq & \unit_{m,m,0, \dots, 0} \oplus \unit_{m,0,m,0 \dots, 0} \oplus \cdots \oplus \unit_{0, \dots, 0,m,m} \\
& \vdots & \\
\Homol_{dm} \left ( C \right ) & \simeq & \unit_{m, \dots, m}
\end{eqnarray*}
Moreover, since for each $j$ the chain map $\nu_j \colon \s^m C_{\zeta_j} \longrightarrow C_{\zeta_j}$ induces an isomorphism $\bar{\nu}_j \colon \Homol_0(C_{\zeta_j}) \longrightarrow \Homol_m(C_{\zeta_j})$, we see that the action of the chain map $\theta_j \colon \s^m C \longrightarrow C$ on $\Homol_*(C)$
is given by
\begin{center}
\begin{tikzpicture}
\diagram{d}{3em}{3em}{
\unit_{i_1, \dots, i_d} & \unit_{i_1, \dots, i_j+m, \dots, i_d}  \\
 };
\path[->, font = \scriptsize, auto]
(d-1-1) edge (d-1-2);
\end{tikzpicture}
\end{center}
with the convention that $\unit_{i_1, \dots, i_j+m, \dots, i_d} =0$ if $i_j =m$. In other words, the induced morphism $\bar{\theta}_j \colon \Homol_{tm}(C) \longrightarrow \Homol_{(t+1)m}(C)$ maps $\unit_{i_1, \dots, i_d}$ isomorphically to $\unit_{i_1, \dots, i_j+m, \dots, i_d}$ if $i_j =0$, and to zero if $i_j =m$. Thus $\Homol_{tm}(C)$ is the direct sum of $\bar{\theta}_{i_1} \cdots \bar{\theta}_{i_t}  \unit_{0, \dots, 0}$, the sum being taken over all $1 \le i_1 < \cdots < i_t \le d$. Since $f$ is an additive function with $f( \unit ) =1$, we see that $f \left ( \Homol_{tm}(C) \right ) =  {{d}\choose{t}}$. 

\sloppy We now adapt the techniques from \cite[Section 2]{IW}. By assumption $d \ge 8$, so let $\Lambda'$ be the exterior subalgebra of $\Lambda$ generated by $\theta_1, \dots, \theta_8$, and $w \in \Lambda'$ the Lefschetz element $\theta_1 \cdot \theta_2 + \theta_3 \cdot \theta_4 + \theta_5 \cdot \theta_6 + \theta_7 \cdot \theta_8$. Moreover, let $C' = C_{\zeta_1} \ot \cdots \ot C_{\zeta_8}$, and $C'' = C_{\zeta_9} \ot \cdots \ot C_{\zeta_d}$, with the convention that $C''$ is the stalk complex on the unit object $\unit$ if $d=8$. By \cite[Proposition A.2]{CHI}, the multiplication map
\begin{center}
\begin{tikzpicture}
\diagram{d}{3em}{3em}{
\Lambda'_i & \Lambda'_{i+2m} \\
 };
\path[->, font = \scriptsize, auto]
(d-1-1) edge node{$w$} (d-1-2);
\end{tikzpicture}
\end{center}
is injective for $i \in \{ 0,m,2m,3m \}$, and surjective for $i \in \{ 3m,4m,5m,6m \}$. We claim that in the same way, the chain map $w \colon \s^{2m} C' \longrightarrow C'$ induces a morphism
\begin{center}
\begin{tikzpicture}
\diagram{d}{3em}{3em}{
\Homol_i \left ( C' \right ) & \Homol_{i+2m} \left ( C' \right )  \\
 };
\path[->, font = \scriptsize, auto]
(d-1-1) edge node{$\bar{w}$} (d-1-2);
\end{tikzpicture}
\end{center}
which is a monomorphism for $i \in \{ 0,m,2m,3m \}$, and an epimorphism for $i \in \{ 3m,4m,5m,6m \}$. To see this, we use the above interpretation of $\Homol_*(C)$ -- which applies to $\Homol_*(C')$ as well -- together with the fact that as a $k$-linear abelian category, $\C$ is equivalent to the category of finitely generated left modules over some finite dimensional algebra; cf.\ \cite[pages 9--10]{EGNO}. 

The epimorphism claim follows from the interpretation of $\Homol_{tm}(C')$ as a direct sum of $\bar{\theta}_{i_1} \cdots \bar{\theta}_{i_t}  \unit_{0, \dots, 0}$, the sum being taken over all $1 \le i_1 < \cdots < i_t \le 8$. Namely, since $w$ induces a surjective map from $\Lambda'_i$ to $ \Lambda'_{i+2m}$ for $i \in \{ 3m,4m,5m,6m \}$, the induced morphism $\bar{w}$ must map $\Homol_{i}(C')$ surjectively onto $\Homol_{i+2m}(C')$ for the same $i$. Note that for $i = 3m$, the map must be an isomorphism, since 
$$\dim_k \Homol_{3m}(C') = {{8}\choose{3}} \dim_k \unit = {{8}\choose{5}} \dim_k \unit = \dim_k \Homol_{5m}(C')$$
(here we have passed to modules over some finite dimensional algebra, so $C'$ and $\unit$ are images under an equivalence of categories). 

For the monomorphism claim, we treat the three cases $i = 0,m,2m$ separately. First, suppose that $x \in \unit_{0, \dots, 0}$ with $0 = \bar{w}(x) = \bar{\theta}_1 \circ \bar{\theta}_2(x) + \cdots + \bar{\theta}_7 \circ \bar{\theta}_8(x)$. Each $\bar{\theta}_s \circ \bar{\theta}_{s+1}(x)$ belongs to the summand $\unit_{0, \dots, s,s+1, \dots, 0}$ of $\Homol_{2m}(C')$, and since $\bar{\theta}_s \circ \bar{\theta}_{s+1}$ maps $\unit_{0, \dots, 0}$ isomorphically onto this summand, we see that $x=0$. Similarly, if $x_1, \dots, x_8$ are elements of $\unit_{0, \dots, 0}$ such that the element $\bar{\theta}_1(x_1) +  \cdots + \bar{\theta}_8(x_8) \in \Homol_{m}(C')$ belongs to the kernel of $\bar{w}$, then it is not hard to see from the decomposition $\Homol_{3m}(C') = \oplus \bar{\theta}_{i_1} \bar{\theta}_{i_2} \bar{\theta}_{i_3} \unit_{0, \dots, 0}$ that $x_1 = \cdots = x_8 =0$. Hence for $i = 0$ and $i=m$, the monomorphism claim follows. Finally, for $i = 2m$, consider the decomposition of $\Homol_{2m}(C')$ as the direct sum of all the $28$ summands $\bar{\theta}_{r} \bar{\theta}_{s} \unit_{0, \dots, 0}$ for $1 \le r < s \le 8$. For all such $r,s$, let $x_{r,s} \in \unit_{0, \dots, 0}$, and suppose that the element $\sum \bar{\theta}_{r} \circ \bar{\theta}_{s} ( x_{r,s} ) \in \Homol_{2m}(C')$ belongs to the kernel of $\bar{w}$. If $s \neq r+1$, then pick a map $\bar{\theta}_u \bar{\theta}_{u+1} \in \{ \bar{\theta}_1 \bar{\theta}_2, \bar{\theta}_3 \bar{\theta}_4, \bar{\theta}_5 \bar{\theta}_6, \bar{\theta}_7 \bar{\theta}_8 \}$ such that the indices $r,s,u,u+1$ are different. Then $\bar{\theta}_u \circ \bar{\theta}_{u+1} \circ \bar{\theta}_{r} \circ \bar{\theta}_{s} ( x_{r,s} )$ belongs to the summand $\bar{\theta}_u \bar{\theta}_{u+1} \bar{\theta}_{r} \bar{\theta}_{s} \unit_{0, \dots, 0}$ of $\Homol_{4m}(C')$, and it is the only term from $\bar{w} \left ( \sum \bar{\theta}_{r} \circ \bar{\theta}_{s} ( x_{r,s} ) \right )$ that belongs to this summand (note that the indices of the summand $\bar{\theta}_u \bar{\theta}_{u+1} \bar{\theta}_{r} \bar{\theta}_{s} \unit_{0, \dots, 0}$ are not necessarily written in the correct ascending order). Consequently, we see that $x_{r,s} =0$ for $s \neq r+1$. The same reasoning works for the elements $x_{2,3}, x_{4,5}, x_{6,7}$. For example, the element $\bar{\theta}_1 \circ \bar{\theta}_2 \circ \bar{\theta}_6 \circ \bar{\theta}_7 ( x_{6,7} )$ is the only term from $\bar{w} \left ( \sum \bar{\theta}_{r} \circ \bar{\theta}_{s} ( x_{r,s} ) \right )$ that belongs to the summand $\bar{\theta}_1 \bar{\theta}_2 \bar{\theta}_6 \bar{\theta}_7 \unit_{0, \dots, 0}$. For the remaining elements $x_{1,2}, x_{3,4}, x_{5,6}, x_{7,8}$, note that the three elements 
$$\bar{\theta}_1 \circ \bar{\theta}_2 \circ \bar{\theta}_3 \circ \bar{\theta}_4 ( x_{1,2} + x_{3,4} ), \bar{\theta}_3 \circ \bar{\theta}_4 \circ \bar{\theta}_5 \circ \bar{\theta}_6 ( x_{3,4} + x_{5,6} ), \bar{\theta}_5 \circ \bar{\theta}_6 \circ \bar{\theta}_1 \circ \bar{\theta}_2 ( x_{5,6} + x_{1,2} )$$
are the only terms from $\bar{w} \left ( \sum \bar{\theta}_{r} \circ \bar{\theta}_{s} ( x_{r,s} ) \right )$ that belong to the corresponding three summands of of $\Homol_{4m}(C')$. Therefore
$$x_{1,2} + x_{3,4} = x_{3,4} + x_{5,6} = x_{5,6} + x_{1,2} =0$$
and so since the characteristic of $k$ is not $2$ we obtain $x_{1,2} = x_{3,4} = x_{5,6} =0$. Finally, this also forces $x_{7,8}$ to be zero. This proves the monomorphism claim for the last case $i=2m$.

Now let $D'$ be the mapping cone of the chain map $w$, and consider the corresponding long exact sequence
\begin{center}
\begin{tikzpicture}
\diagram{d}{3em}{3em}{
\cdots &\Homol_i \left ( C' \right )  & \Homol_i \left ( D' \right ) & \Homol_{i-2m-1} \left ( C' \right ) & \Homol_{i-1} \left ( C' \right ) & \cdots \\
 };
\path[->, font = \scriptsize, auto]
(d-1-1) edge node{$\bar{w}$} (d-1-2)
(d-1-2) edge (d-1-3)
(d-1-3) edge (d-1-4)
(d-1-4) edge node{$\bar{w}$} (d-1-5)
(d-1-5) edge (d-1-6);
\end{tikzpicture}
\end{center}
in homology. From the fact that $\Homol_*(C')$ is nonzero only in degrees $0,m,2m, \dots, 8m$, we obtain isomorphisms
\begin{eqnarray*}
\Homol_0 \left ( D' \right ) & \simeq & \Homol_0 \left ( C' \right ) \\
\Homol_m \left ( D' \right ) & \simeq & \Homol_m \left ( C' \right ) \\
\Homol_{9m+1} \left ( D' \right ) & \simeq & \Homol_{7m} \left ( C' \right ) \\
\Homol_{10m+1} \left ( D' \right ) & \simeq & \Homol_{8m} \left ( C' \right )
\end{eqnarray*}
and short exact sequences
\begin{center}
\begin{tikzpicture}
\diagram{d}{3em}{3em}{
0 & \Homol_{im} \left ( C' \right ) &  \Homol_{(i+2)m} \left ( C' \right ) &  \Homol_{(i+2)m} \left ( D' \right ) & 0 \\
0 & \Homol_{(6+i)m+1} \left ( D' \right ) &  \Homol_{(4+i)m} \left ( C' \right ) &  \Homol_{(6+i)m} \left ( C' \right ) & 0 \\
 };
\path[->, font = \scriptsize, auto]
(d-1-1) edge (d-1-2)
(d-1-2) edge node{$\bar{w}$} (d-1-3)
(d-1-3) edge (d-1-4)
(d-1-4) edge (d-1-5)
(d-2-1) edge (d-2-2)
(d-2-2) edge (d-2-3)
(d-2-3) edge node{$\bar{w}$} (d-2-4)
(d-2-4) edge (d-2-5);
\end{tikzpicture}
\end{center}
for $i \in \{ 0,1,2 \}$. All the other homology objects of $D'$ are zero. By combining all of this with the fact that $f \left ( \Homol_{im}(C') \right ) =  {{8}\choose{i}}$, we can compute 
$$\sum_{i \in \mathbb{Z}} f \left ( \Homol_i(D') \right ) = {{8}\choose{3}} + {{8}\choose{4}} + {{8}\choose{4}} + {{8}\choose{5}} = 2^8 -4$$

Since $\K^b ( \C )$ is a triangulated tensor category with an exact tensor product, we can apply $- \ot C''$ to the distinguished triangle
\begin{center}
\begin{tikzpicture}
\diagram{d}{3em}{3em}{
\s^{2m} C' & C' & D' & \s^{2m+1} C' \\
 };
\path[->, font = \scriptsize, auto]
(d-1-1) edge node{$w$} (d-1-2)
(d-1-2) edge (d-1-3)
(d-1-3) edge (d-1-4);
\end{tikzpicture}
\end{center}
and obtain a distinguished triangle
\begin{center}
\begin{tikzpicture}
\diagram{d}{3em}{3em}{
\s^{2m}C & C & D' \ot C'' & \s^{2m+1} C \\
 };
\path[->, font = \scriptsize, auto]
(d-1-1) edge node{$w$} (d-1-2)
(d-1-2) edge (d-1-3)
(d-1-3) edge (d-1-4);
\end{tikzpicture}
\end{center}
where we view $w$ as a chain map $\s^{2m}C \longrightarrow C$ via the inclusion $\Lambda' \subseteq \Lambda$. Let $D$ be the mapping cone of this chain map. Since all the objects in $C$ are projective, so are all the objects in $D$. Moreover, it is homotopy equivalent to the complex $D' \ot C''$, giving
\begin{eqnarray*}
\sum_{i \in \mathbb{Z}} f \left ( \Homol_i(D) \right ) & = & \sum_{i \in \mathbb{Z}} f \left ( \Homol_i(D' \ot C'') \right ) \\
& = & \sum_{i \in \mathbb{Z}} f \left ( \bigoplus_{s+t=i} \Homol_s(D') \ot \Homol_t(C'') \right ) \\
& = & \sum_{i \in \mathbb{Z}} f \left (  \bigoplus_{t=0}^{d-8} \Homol_{i-tm}(D') \ot \Homol_{tm}(C'') \right ) \\
& = & \sum_{i \in \mathbb{Z}} \left ( \sum_{t=0}^{d-8} {{d-8}\choose{t}} f \left ( \Homol_{i-tm}(D') \right ) \right ) \\
& = & \left (  \sum_{i \in \mathbb{Z}} f \left ( \Homol_i(D') \right ) \right ) \left ( \sum_{t=0}^{d-8} {{d-8}\choose{t}} \right ) \\
& = & \left ( 2^8 -4 \right ) 2^{d-8} \\
& < & 2^d
\end{eqnarray*}
Here we have used the K{\"u}nneth Theorem, the fact that the homology of the complex $C''$ is concentrated in degrees $0,m,2m, \dots, (d-8)m$, and the fact that $\Homol_{tm}(C'')$ is the direct sum of ${{d-8}\choose{t}}$ copies of the unit object $\unit$.

We have now constructed one complex $D$, as the mapping cone of the chain map $\s^{2(n-1)}C \longrightarrow C$ given by $w$, where $n$ was the common even degree of the homogeneous elements $\zeta_1, \dots, \zeta_d$ in $\Ho ( \C )$. The length of the complex $C$ is $d(n-1)+1$; it has nonzero objects in degrees $0,1, \dots, d(n-1)$. The complex $D$ therefore has nonzero objects in degrees $0,1, \dots, (d+2)(n-1)+1$, and therefore length $(d+2)(n-1)+2$. Now take any positive integer $s$. If we go back to the beginning, and replace the elements $\zeta_1, \dots, \zeta_d$ by $\zeta_1^s, \dots, \zeta_d^s$, we obtain a complex $D^s$ of length $(d+2)(sn-1)+2$. Thus for each $s \ge 1$, we obtain a complex $D^s$ of projective objects, with $f \left ( \sum \Homol_i(D^s) \right ) < 2^d$. Moreover, when $s \neq t$ then the complexes $D^s$ and $D^t$ are not isomorphic.
\end{proof}

The length function $\ell$ is additive on $\C$, with $\ell ( \unit ) = 1$, so the theorem applies to this function. However, there is another additive function of great importance as well. Let $S_1, \dots, S_t$ be a complete set of representatives of the isomorphism classes of simple objects in $\C$; there are only finitely many by assumption. Since the unit object is simple, we can assume that $S_1 = \unit$, but it makes no difference in what follows. Given an object $M \in \C$, we denote by $[M:S_i]$ the multiplicity of the simple object $S_i$ in a Jordan-H{\"o}lder series of $M$; this is well defined since the multiplicity is the same in all such series. Given a simple object $S$, consider the matrix of left multiplication by $S$, that is, the $t \times t$-matrix whose $ij$-entry is $[S \ot S_i : S_j]$. This matrix has non-negative integer entries, and therefore a non-negative real eigenvalue by the Frobenius-Perron Theorem. The maximal real eigenvalue is the \emph{Frobenius-Perron dimension} of $S$, and denoted by $\FPdim_{\C}(S)$. It is at least $1$. For details, we refer to \cite[Chapter 3]{EGNO}

One extends the definition of the Frobenius-Perron dimension to arbitrary objects by additivity. Namely, given $M \in \C$, we set
$$\FPdim_{\C}(M) = \sum_{i=1}^t [M : S_i] \FPdim_{\C}(S_i)$$
We see immediately that $\FPdim_{\C}(M) \ge 1$ whenever $M$ is a nonzero object, since the inequality holds for all the simple objects. Also, we see that $\FPdim_{\C}(M) \ge \ell (M)$. Moreover, given a short exact sequence
\begin{center}
\begin{tikzpicture}
\diagram{d}{3em}{3em}{
0 & L & M & N & 0 \\
 };
\path[->, font = \scriptsize, auto]
(d-1-1) edge (d-1-2)
(d-1-2) edge (d-1-3)
(d-1-3) edge (d-1-4)
(d-1-4) edge (d-1-5);
\end{tikzpicture}
\end{center}
of objects, there is an equality of multiplicities $[M : S] = [L : S] + [N : S]$ for every simple object $S$, and so
$$\FPdim_{\C}(M) = \FPdim_{\C}(L) + \FPdim_{\C}(N)$$
that is, the Frobenius-Perron dimension is additive in $\C$. Finally, since the matrix of left multiplication by the unit object $\unit$ is just $[S_i : S_j]$, that is, the $t \times t$-identity matrix, we see that $\FPdim_{\C}( \unit ) = 1$.

\begin{corollary}\label{cor:mainFPdim}
Let $k$ be a field of characteristic not $2$, and $\left ( \C, \ot, \unit \right )$ a finite tensor $k$-category. Suppose that $\C$ satisfies \emph{\textbf{Fg}}, and that the Krull dimension $d$ of $\Ho ( \C )$ is at least $8$. Then there exist infinitely many non-isomorphic and nontrivial bounded complexes $D \in \K^b ( \C )$ of projective objects, with 
$$\sum_{i \in \mathbb{Z}} \ell ( \Homol_i(D) ) \le \sum_{i \in \mathbb{Z}} \FPdim_{\C} ( \Homol_i(D) ) < 2^d$$
\end{corollary}

When $\C$ is the category $\mod A$ of finitely generated left modules over a finite dimensional Hopf algebra $A$, then the Frobenius-Perron dimension coincides with the ordinary vector space dimension, as explained in \cite[Example 4.5.5]{EGNO} in the case of group algebras. We therefore obtain the following corollary.

\begin{corollary}\label{cor:Hopf}
Let $k$ be a field of characteristic not $2$, and $A$ a finite dimensional Hopf algebra over $k$. Suppose that $A$ satisfies \emph{\textbf{Fg}}, and that the Krull dimension $d$ of the even cohomology ring $\Ho (A) = \Coh^{2*}(A)$ is at least $8$. Then there exist infinitely many non-isomorphic and nontrivial  bounded complexes $D \in \K^b ( \mod A )$ of projective left modules, with 
$$\sum_{i \in \mathbb{Z}} \dim_k \Homol_i(D) < 2^d$$
\end{corollary}

In particular, we recover the following result for group algebras of finite groups, recently proved by Carlson in \cite{Ca}. Recall that for a prime $p$ and a finite group $G$, the $p$-rank of $G$ is the maximal rank of an elementary abelian $p$-subgroup of $G$. By a result of Quillen (cf.\ \cite{Q}), this equals the Krull dimension of the (even part of the) cohomology ring $\Coh^* (G,k)$ when the ground field $k$ has characteristic $p$. Recall also that when $G$ is a $p$-group, then every projective module is free.

\begin{corollary}\label{cor:groups}\cite[Theorem]{Ca}
Let $k$ be a field of characteristic $p \ge 3$, and $G$ a finite group of $p$-rank $d \ge 8$. Then there exist infinitely many non-isomorphic and nontrivial  bounded complexes $D \in \K^b ( \mod kG )$ of projective left modules, with 
$$\sum_{i \in \mathbb{Z}} \dim_k \Homol_i(D) < 2^d$$
In particular, when $G$ is a $p$-group, then the complexes are complexes of free modules.
\end{corollary}

\section{Finite dimensional algebras}\label{sec:findimalg}

In this section, we briefly sketch a version of Theorem \ref{thm:main} for certain algebras. As before, we fix a field $k$, not necessarily algebraically closed, and we take a finite dimensional indecomposable $k$-algebra $A$. All modules are assumed to be finitely generated; we denote by $\mod A$ the category of left such modules.

\sloppy As for finite tensor categories, there is a theory of support varieties for left $A$-modules; we refer to \cite{EHSST, SS04,S} for details. Let $\Hoch^*(A)$ be the Hochschild cohomology ring of $A$, i.e.\ $\Hoch^n(A) = \Ext_{A^e}^n(A,A)$, where $A^e = A \ot_k A^{\op}$ is the enveloping algebra. By a classical result of Gerstenhaber (cf.\ \cite{G}), this ring is graded-commutative. For every $M \in \mod A$, there is a homomorphism
\begin{center}
\begin{tikzpicture}
\diagram{d}{3em}{4em}{
\Hoch^*(A) & \Ext_A^*(M,M) \\
 };
\path[->, font = \scriptsize, auto]
(d-1-1) edge node{$- \ot_A M$} (d-1-2);
\end{tikzpicture}
\end{center}
of graded $k$-algebras, and the corresponding left and right $\Hoch^*(A)$-module structures on $\Ext_A^*(M,M)$ coincide up to a sign. Now define $H$ to be $\Hoch^{2*}(A)$ when the characteristic of $k$ is not $2$, and $\Hoch^*(A)$ when the characteristic is $2$. Note that since $A$ is indecomposable as a $k$-algebra, its center $Z(A)$ is a commutative local ring, say with maximal ideal $\n$. As $\Hoch^0(A) = Z(A)$, we see that $H$ has a unique homogeneous maximal ideal $\m_0 = \n \oplus H \cap \Hoch^{\ge 1}(A)$. The \emph{support variety} of $M$ is
$$V_H(M) = \{ \m_0 \} \cup \{ \m \in \Maxspec H \mid \Ann_H \Ext_A^*(M,M) \subseteq \m \}$$

\begin{definition}
The algebra $A$ satisfies \textbf{Fg} if $\Hoch^*(A)$ is finitely generated, and $\Ext_A^*(M,M)$ is a finitely generated $\Hoch^*(A)$-module for every $M \in \mod A$.
\end{definition}

As explained in the proof of \cite[Proposition 5.7]{S}, the finiteness condition is equivalent to the following: $H$ is finitely generated, and $\Ext_A^*(M,M)$ is a finitely generated $H$-module for every $A$-module $M$. Moreover, it suffices to require that the $H$-module $\Ext_A^*(A/ \mathfrak{r}_A,A/ \mathfrak{r}_A)$ be finitely generated, where $\mathfrak{r}_A$ is the radical of $A$. As shown in \cite{EHSST}, when \textbf{Fg} holds, then as for finite tensor categories, many of the results on support varieties over group algebras carry over.

\begin{remark}\label{rem:fg}
If $A$ is a finite dimensional Hopf algebra over $k$, then we now have two different versions of the condition \textbf{Fg}; one in terms of the cohomology ring $\Coh^*(A)$ and $\mod A$ viewed as a finite tensor category, and one in terms of the Hochschild cohomology ring $\HH^*(A)$ and $\mod A$ viewed as just an ordinary module category. However, by \cite[Lemma 4.2]{B2} they are equivalent. Consequently,  Corollary \ref{cor:Hopf} can also be deduced from the results in this section.
\end{remark}

We now construct complexes as in Construction \ref{con:complex}. Fix a minimal projective bimodule resolution $(Q_n, d_n)$ of $A$, and let $\eta \in \Hoch^*(A)$ be a nonzero homogeneous element of degree $n \ge 2$. This element corresponds to a map $\hat{\eta} \colon \Omega_{A^e}^n(A) \longrightarrow A$, and the pushout with the inclusion $\Omega_{A^e}^n(A) \longrightarrow Q_{n-1}$ gives a commutative diagram
\begin{center}
\begin{tikzpicture}
\diagram{d}{2.8em}{2.8em}{
0 & \Omega_{A^e}^n ( A ) & Q_{n-1} & Q_{n-2} & \cdots & Q_0 & A & 0 \\
0 & A & M_{\eta} & Q_{n-2} & \cdots & Q_0 & A & 0 \\
 };
\path[->, font = \scriptsize, auto]
(d-1-1) edge (d-1-2) 
(d-1-2) edge (d-1-3) 
(d-1-3) edge node{$d_{n-1}$} (d-1-4)
(d-1-4) edge node{$d_{n-2}$} (d-1-5)
(d-1-5) edge node{$d_1$} (d-1-6)
(d-1-6) edge node{$d_0$} (d-1-7)
(d-1-7) edge (d-1-8)
(d-2-1) edge (d-2-2) 
(d-2-2) edge node{$\mu$} (d-2-3) 
(d-2-3) edge node{$\rho$} (d-2-4)
(d-2-4) edge node{$d_{n-2}$} (d-2-5)
(d-2-5) edge node{$d_1$} (d-2-6)
(d-2-6) edge node{$d_0$} (d-2-7)
(d-2-7) edge (d-2-8)
(d-1-2) edge node{$\hat{\eta}$} (d-2-2)
(d-1-3) edge (d-2-3)
(d-1-4) edge node{$1$} (d-2-4)
(d-1-6) edge node{$1$} (d-2-6)
(d-1-7) edge node{$1$} (d-2-7);
\end{tikzpicture}
\end{center}
with exact rows. We denote the corresponding bimodule complex
\begin{center}
\begin{tikzpicture}
\diagram{d}{2.8em}{2.8em}{
\cdots & 0 & M_{\eta} & Q_{n-2} & \cdots & Q_0 & 0 & \cdots \\
 };
\path[->, font = \scriptsize, auto]
(d-1-1) edge (d-1-2) 
(d-1-2) edge (d-1-3) 
(d-1-3) edge node{$\rho$} (d-1-4) 
(d-1-4) edge node{$d_{n-2}$} (d-1-5) 
(d-1-5) edge node{$d_1$} (d-1-6)
(d-1-6) edge (d-1-7)
(d-1-7) edge (d-1-8);
\end{tikzpicture}
\end{center}
by $C_{\eta}$. As for the complexes we used earlier, this one has two nonzero homology modules $\Homol_0 ( C_{\eta} ) \simeq \Homol_{n-1} ( C_{\eta} ) \simeq A$, and the natural chain map $\nu \colon \Sigma^{n-1} C_{\eta} \longrightarrow C_{\eta}$ given by the composition $\mu \circ d_0$ induces an isomorphism $\bar{\nu}$ between these. 

\begin{remark}\label{rem:projective}
Every projective bimodule is projective as a one-sided (left or right) $A$-module. This can be seen by noting that the free bimodule $A^e$ is left and right projective. Now consider the bottom exact sequence in the above diagram. Since $A$ is obviously left and right projective, all the underlying short exact sequences split. Therefore all the modules in the complex $C_{\eta}$ are one-sided projective, and the same holds for all the cycles and the boundaries.
\end{remark}

Let $\P$ be the full subcategory of $\mod A^e$ containing the modules that are projective as one-sided (left and right) $A$-modules. This is an additive (even exact) subcategory of $\mod A^e$, with a monoidal structure induced by the tensor product over $A$. For if $B_1$ and $B_2$ are modules in $\P$, that is, one-sided projective bimodules, then $B_1 \ot_A B_2$ is also one-sided projective, by the hom-tensor adjunction. The category $\K^b ( \P )$ is a triangulated tensor category, with an exact tensor product. Moreover, it acts on the triangulated category $\K^b ( \mod A )$ in the sense of \cite{BKSS}. This action is exact, that is, given a complex $C \in \K^b ( \mod A )$, the functor 
\begin{center}
\begin{tikzpicture}
\diagram{d}{3em}{4em}{
\K^b ( \P ) & \K^b ( \mod A ) \\
 };
\path[->, font = \scriptsize, auto]
(d-1-1) edge node{$- \ot_A C$} (d-1-2);
\end{tikzpicture}
\end{center}
is exact. It is also exact in the other argument, as a triangulated endofunctor on $\K^b ( \mod A )$.

We now prove the analogue of Theorem \ref{thm:main}. We denote by $\cx_A (M)$ the complexity of an $A$-module $M$, that is, the rate of growth of the dimensions of the modules in its minimal projective resolution. The maximal complexity is obtained by $A / \mathfrak{r}_A$ (and therefore by one of the simple modules), and there is an inequality $\cx_A ( A / \mathfrak{r}_A ) \le \dim H$ when \textbf{Fg} holds. As explained in the proof of \cite[Corollary 3.6]{B}, equality holds whenever $A / \mathfrak{r}_A \ot_k A / \mathfrak{r}_A$ is a semisimple ring, for example when $k$ is algebraically closed.

\begin{theorem}\label{thm:findimmain}
Let $k$ be a field of characteristic not $2$, and $A$ a finite dimensional indecomposable selfinjective $k$-algebra. Suppose that $A$ satisfies \emph{\textbf{Fg}}, and that  there exists a simple $A$-module $S$ of complexity $c \ge 8$. Then there exist infinitely many non-isomorphic and nontrivial bounded complexes $D \in \K^b ( \mod A )$ of projective modules, with 
$$\frac{1}{\dim_k S} \sum_{i \in \mathbb{Z}} \dim_k \Homol_i(D)  < 2^c$$
In particular, if there exists a simple one-dimensional module of complexity $c \ge 8$, then there exist infinitely many non-isomorphic and nontrivial such complexes $D$ with
$$\sum_{i \in \mathbb{Z}} \dim_k \Homol_i(D)  < 2^c$$
\end{theorem}

\begin{proof}
The proof is similar to that of Theorem \ref{thm:main}, so we only give a sketch of the construction of one such complex $D$. The last part of the proof of Theorem \ref{thm:main} then applies, and gives a family of other complexes. We shall apply some of the results from \cite{EHSST}, in which it is assumed that the ground field is algebraically closed. However, the proofs of the results that we use do not require this assumption.

Let $H =  \Hoch^{2*}(A)$, and denote the homogeneous ideal $\Ann_H \Ext_A^*(M,M)$ by $\az$. By \cite[Proposition 2.1]{EHSST}, the Krull dimension of $H / \az$ is $c$, and so it contains a homogeneous system of parameters with $c$ elements. When we lift these to $H$ we obtain homogeneous elements $\eta_1, \dots, \eta_c$ of positive degrees, and we may suppose that they are of the same even degree $n$; we denote the odd integer $n-1$ by $m$.

For each $1 \le i \le 8$, consider the bimodule complex $C_{\eta_i}$ and the chain map $\nu_i$ that we defined earlier in this section (recall that $c \ge 8$). By Remark \ref{rem:projective}, the cycles and boundaries in $C_{\eta_i}$ are projective as one-sided modules, and so are the two nonzero homology modules $\Homol_0 ( C_{\eta_i} )$ and $\Homol_{m} ( C_{\eta_i} )$, since these are both isomorphic to $A$. Therefore, by \cite[Theorem 2.7.1 and Corollary 2.7.2]{Ben}, the K{\"u}nneth Theorem holds whenever we tensor $C_{\eta_i}$ with a complex in $\K^b ( \P )$ or $\K^b ( \mod A )$. Now consider the complex $C_{\eta_8} \ot_A \cdots \ot_A C_{\eta_1} \ot_A S$ in $\K^b ( \mod A )$; let us denote it by $C'$. It has homology in degrees $0,m,2m, \dots, 8m$, and $\Homol_{tm} (C')$ is the direct sum of ${{8}\choose{t}}$ copies of $S$.

The chain maps $\nu_i$ induce chain maps on $C_{\eta_8} \ot_A \cdots \ot_A C_{\eta_1}$, and in turn chain maps $\theta_i \colon \Sigma^m C' \longrightarrow C'$. These generate an $8$-fold exterior algebra in $\Hom_{\K^b( \mod A)}^*(C,C)$, and the Lefschetz element $w = \theta_1 \cdot \theta_2 +  \theta_3 \cdot \theta_4 +  \theta_5 \cdot \theta_6 +  \theta_7 \cdot \theta_8$ induces a homomorphism
\begin{center}
\begin{tikzpicture}
\diagram{d}{3em}{3em}{
\Homol_i \left ( C' \right ) & \Homol_{i+2m} \left ( C' \right )  \\
 };
\path[->, font = \scriptsize, auto]
(d-1-1) edge node{$\bar{w}$} (d-1-2);
\end{tikzpicture}
\end{center}
which is injective for $i \in \{ 0,m,2m,3m \}$, and surjective for $i \in \{ 3m,4m,5m,6m \}$. The homology of the mapping cone $D'$ of $w$ then satisfies
$$\sum_{i \in \mathbb{Z}} \dim_k \Homol_i \left ( D' \right ) = \left ( 2^8-4 \right ) \dim_k S$$

Now let $C'' = C_{\eta_c} \ot_A \cdots \ot_A C_{\eta_9}$, with the convention that this is the stalk complex on $A$ if $c=8$, and set $C = C'' \ot_A C'$. From the distinguished triangle on the chain map $w$, we obtain a distinguished triangle
\begin{center}
\begin{tikzpicture}
\diagram{d}{3em}{3em}{
\s^{2m}C & C &  C'' \ot_A D' & \s^{2m+1} C \\
 };
\path[->, font = \scriptsize, auto]
(d-1-1) edge node{$w$} (d-1-2)
(d-1-2) edge (d-1-3)
(d-1-3) edge (d-1-4);
\end{tikzpicture}
\end{center}
in $\K^b ( \mod A )$, where we have denoted the chain map $1 \ot_A w$ on $C = C'' \ot_A C'$ by $w$ as well. Let $D$ be the mapping cone of this chain map. The complex $C''$ has nonzero homology in degrees $0,m,2m, \dots, (c-8)m$, with $\Homol_{tm} (C'')$ the direct sum of ${{c-8}\choose{t}}$ copies of $A$. This gives
$$\sum_{i \in \mathbb{Z}} \dim_k \Homol_i(D) = \sum_{i \in \mathbb{Z}} \dim_k \Homol_i(C'' \ot_A D') = 2^{c-8} \left ( 2^8-4 \right ) \dim_k S < 2^c \dim_k S$$

The complex $D$ consists of projective modules. To see this, note that each nonzero term in $C$ is a direct sum of modules of the form $B_1 \ot_A \cdots \ot_A B_c \ot_A C$, where the bimodules $B_i$ appear in the complexes $\eta_1, \dots, \eta_c$. Suppose that at least one of the $B_i$ is a projective bimodule $Q_j$. In general, if $Q$ is a projective bimodule, and $M$ is a left $A$-module, then $Q \ot_A M$ is a projective $A$-module; this follows from the fact that $A^e \ot_A M$ is projective. Now it follows that if $B_1$ is projective, then $B_1 \ot_A \cdots \ot_A B_c \ot_A C$ is projective. Otherwise, the tensor product is of the form $M_{\eta_{i_1}} \ot_A \cdots \ot_A M_{\eta_{i_s}} \ot_A Q \ot_A \cdots \ot_A C$, where $M_{\eta_i}$ is the leftmost nonzero bimodule in the complex $C_{\eta_i}$, and $Q$ is a projective bimodule. Since the $A$-module $Q \ot_A \cdots \ot_A C$ is projective, and the bimodule $M_{\eta_{i_1}} \ot_A \cdots \ot_A M_{\eta_{i_s}}$ is one-sided projective, it follows from the hom-tensor adjunction that $M_{\eta_{i_1}} \ot_A \cdots \ot_A M_{\eta_{i_s}} \ot_A Q \ot_A \cdots \ot_A C$ is projective as well. Finally, the only module in $C$ that is not a direct sum of modules of the above type is the leftmost nonzero module, namely $M_{\eta_c} \ot_A \cdots \ot_A M_{\eta_1} \ot_A S$. However, by \cite[Proposition 4.3]{EHSST}, the support variety of this module is $Z( \eta_1, \dots \eta_c ) \cap V_H(S)$, which equals $Z( \eta_1, \dots \eta_c ) \cap Z( \az )$. By construction, the images of the $\eta_i$ form a homogeneous system of parameters in $H / \az$, and so the support variety of the module is trivial. Therefore, by \cite[Proposition 2.2]{EHSST} and the fact that $A$ is selfinjective, the module is projective. This shows that the complex $C$, and therefore also $D$, consists of projective modules. 
\end{proof}

When $A$ is a local algebra, then it is indecomposable, and every projective module is free. Moreover, if such an algebra satisfies \textbf{Fg}, then it is selfinjective, since it is Gorenstein by \cite[Proposition 2.2]{EHSST}. Therefore, the theorem then takes the following form.

\begin{corollary}\label{cor:findimlocal}
Let $k$ be a field of characteristic not $2$, and $A$ a finite dimensional local $k$-algebra with maximal left ideal $\m$. Suppose that $A$ satisfies \emph{\textbf{Fg}}, and that $\cx_A ( A / \m ) = c \ge 8$. Then there exist infinitely many non-isomorphic and nontrivial bounded complexes $D \in \K^b ( \mod A )$ of free modules, with 
$$\frac{1}{\dim_k A/ \m} \sum_{i \in \mathbb{Z}} \dim_k \Homol_i(D)  < 2^c$$
In particular, if $\dim_k A/ \m =1$, then there exist infinitely many non-isomorphic and nontrivial such complexes $D$ with
$$\sum_{i \in \mathbb{Z}} \dim_k \Homol_i(D)  < 2^c$$
\end{corollary}

We end with an example of a class of local algebras satisfying the assumptions of the last result.

\begin{example}
Let $k$ be a field, and fix integers $c \ge 1$ and $a_1, \dots, a_c \ge 2$, together with nonzero elements $q_{ij} \in k$ for $1 \le i<j \le c$. The corresponding \emph{quantum complete intersection} is the algebra
$$A = k \langle x_1, \dots, x_c \rangle / ( x_i^{a_i}, x_jx_i - q_{ij} x_ix_j )$$
which is finite dimensional, local and selfinjective. The simple module $k$ has complexity $c$. By \cite[Theorem 5.5]{BO}, this algebra satisfies \textbf{Fg} if and only if all the commutators $q_{ij}$ are roots of unity in $k$. Therefore, when this holds, the characteristic of $k$ is not $2$, and $c \ge 8$, there exist infinitely many non-isomorphic and nontrivial complexes $D$ of free modules, with
$$\sum_{i \in \mathbb{Z}} \dim_k \Homol_i(D)  < 2^c$$
Examples of such algebras are exterior algebras
$$k \langle x_1, \dots, x_c \rangle / ( x_i^2, x_jx_i + x_ix_j )$$
and finite dimensional commutative local complete intersections
$$k [x_1, \dots, x_c] / ( x_1^{a_1}, \dots, x_c^{a_c} )$$
for $c \ge 8$. Note that when we computed the dimension of the homology in the proof of Theorem \ref{thm:findimmain}, we actually obtained the precise number $2^c - 2^{c-6}$ in the case when the simple module is one-dimensional. Therefore, for the complete intersections above, we obtain (infinitely many) complexes of free modules as in \cite[Theorem 3.1]{IW}.
\end{example}


\begin{thebibliography}{999}

\bibitem{AS}
A.\ Adem, R.G.\ Swan, \emph{Linear maps over abelian group algebras}, J.\ Pure Appl.\ Algebra 104 (1995), no.\ 1, 1--7. 

\bibitem{Ben}
D.J.\ Benson, \emph{Representations and cohomology. I. Basic representation theory of finite groups and associative algebras}, Second edition, Cambridge Studies in Advanced Mathematics, 30, Cambridge University Press, Cambridge, 1998, xii+246 pp.

\bibitem{BC}
D.J.\ Benson, J.F.\ Carlson, \emph{Projective resolutions and Poincar{\'e} duality complexes}, Trans.\ Amer.\ Math.\ Soc.\ 342 (1994), no.\ 2, 447--488.

\bibitem{BE1}
D.J.\ Benson, P.\ Etingof, \emph{Symmetric tensor categories in characteristic $2$}, Adv.\ Math.\ 351 (2019), 967--999.

\bibitem{BE2}
D.J.\ Benson, P.\ Etingof, \emph{On cohomology in symmetric tensor categories in prime characteristic}, Homology Homotopy Appl.\ 24 (2022), no.\ 2, 163--193.

\bibitem{BEO}
D.J.\ Benson, P.\ Etingof, V.\ Ostrik, \emph{New incompressible symmetric tensor categories in positive characteristic}, Duke Math.\ J.\ 172 (2023), no.\ 1, 105--200.

\bibitem{B}
P.A.\ Bergh, \emph{Representation dimension and finitely generated cohomology}, Adv.\ Math.\ 219 (2008), no.\ 1, 389--400. 

\bibitem{B2}
P.A.\ Bergh, \emph{Separable equivalences, finitely generated cohomology and finite tensor categories}, preprint.

\bibitem{BO}
P.A.\ Bergh, S.\ Oppermann, \emph{Cohomology of twisted tensor products}, J.\ Algebra 320 (2008), no.\ 8, 3327--3338.

\bibitem{BPW1}
P.A.\ Bergh, J.\ Plavnik, S.\ Witherspoon, \emph{Support varieties for finite tensor categories:
complexity, realization, and connectedness}, J.\ Pure Appl.\ Algebra 225 (2021), no.\ 9, Paper No.\ 106705, 21 pp.

\bibitem{BPW2}
P.A.\ Bergh, J.\ Plavnik, S.\ Witherspoon, \emph{Support varieties for finite tensor categories: the tensor product property}, in progress. 

\bibitem{BH}
W.\ Bruns, J.\ Herzog, \emph{Cohen-Macaulay rings}, Cambridge Studies in Advanced Mathematics, 39, Cambridge University Press, Cambridge, 1993, xii+403 pp.

\bibitem{BKSS}
A.B.\ Buan, H.\ Krause, N.\ Snashall, {\O}.\ Solberg, \emph{Support varieties -- an axiomatic approach}, Math.\ Z.\ 295 (2020), no.\ 1--2, 395--426.

\bibitem{Ca}
J.F.\ Carlson, \emph{The ranks of homology of complexes of projective modules over finite groups}, to appear in Proc.\ Amer.\ Math.\ Soc.

\bibitem{GCa1}
G.\ Carlsson, \emph{Free $(\mathbb{Z}/2)^k$-actions and a problem in commutative algebra}, Transformation groups, Pozna{\'n} 1985, 79--83, Lecture Notes in Math., 1217, Springer, Berlin, 1986.

\bibitem{GCa2}
G.\ Carlsson, \emph{Free $(\mathbb{Z}/2)^3$-actions on finite complexes}, in \emph{Algebraic topology and algebraic K-theory (Princeton, N.J., 1983)}, 332--344, Ann.\ of Math.\ Stud., 113, Princeton Univ.\ Press, Princeton, NJ, 1987. 

\bibitem{CHI}
A.\ Conca, H.-C.\ Herbig, S.B.\ Iyengar, Srikanth, \emph{Koszul properties of the moment map of some classical representations}, Collect.\ Math.\ 69 (2018), no.\ 3, 337--357.

\bibitem{C}
K.\ Coulembier, \emph{Monoidal abelian envelopes}, Compos.\ Math.\ 157 (2021), no.\ 7, 1584--1609.

\bibitem{EHSST}
K.\ Erdmann, M.\ Holloway, N.\ Snashall, {\O}.\ Solberg, R.\ Taillefer, \emph{Support varieties for selfinjective algebras}, K-Theory 33 (2004), no.\ 1, 67--87.

\bibitem{EGNO} 
P.\ Etingof, S.\ Gelaki, D.\ Nikshych, V.\ Ostrik, {\em Tensor categories}, Mathematical Surveys and Monographs 205, American Mathematical Society, Providence, RI, 2015.
 
\bibitem{EO} 
P.\ Etingof, V.\ Ostrik, {\em Finite tensor categories}, Mosc.\ Math.\ J.\ 4 (2004), no.\ 3, 627--654, 782--783. 

\bibitem{E}
L.\ Evens, \emph{The cohomology ring of a finite group}, Trans.\ Amer.\ Math.\ Soc.\ 101 (1961), 224--239.

\bibitem{G}
M.\ Gerstenhaber, \emph{The cohomology structure of an associative ring}, Ann.\ of Math.\ (2) 78 (1963), 267--288.

\bibitem{IW}
S.B.\ Iyengar, M.E.\ Walker, \emph{Examples of finite free complexes of small rank and small homology}, Acta Math.\ 221 (2018), no.\ 1, 143--158.

\bibitem{Q}
D.\ Quillen, \emph{The spectrum of an equivariant cohomology ring. I, II}, Ann.\ of Math.\ (2) 94 (1971), 549--572; ibid.\ (2) 94 (1971), 573--602.

\bibitem{RS}
H.\ R{\"u}ping, M.\ Stephan, \emph{Multiplicativity and nonrealizable equivariant chain complexes}, J.\ Pure Appl.\ Algebra 226 (2022), no.\ 8, Paper No.\ 107023, 36 pp.

\bibitem{SS04}
N.\ Snashall, {\O}.\ Solberg, {\em Support varieties and Hochschild cohomology rings}, Proc.\ London Math.\ Soc.\ 88 (2004), no.\ 3, 705--732. 

\bibitem{S}
{\O}.\ Solberg, \emph{Support varieties for modules and complexes}, in \emph{Trends in representation theory of algebras and related topics}, 239--270, Contemp.\ Math., 406, Amer.\ Math.\ Soc., Providence, RI, 2006. 

\bibitem{SA}
M.\ Su\'arez-\'Alvarez, {\em The Hilton-Eckmann argument for the anti-commutativity of cup products}, Proc.\ Amer.\ Math.\ Soc.\ 132 (2004), no.\ 8, 2241--2246.

\bibitem{V1}
B.B.\ Venkov, \emph{Cohomology algebras for some classifying spaces} (Russian), Dokl.\ Akad.\ Nauk SSSR 127 (1959), 943--944.

\bibitem{V2}
B.B.\ Venkov, \emph{Characteristic classes for finite groups} (Russian), Dokl.\ Akad.\ Nauk SSSR 137 (1961), 1274--1277.

\end{thebibliography}
\end{document}